\documentclass[12pt]{article}
\usepackage{amsfonts}
\usepackage{amssymb}
\usepackage{mathrsfs}
\usepackage{srcltx}
\usepackage{dsfont,cases}
\usepackage{graphicx}
\usepackage{color}
\usepackage{amsthm}
\usepackage{amstext}
\usepackage{amsopn}
\usepackage{graphicx}

\newtheorem{theorem}{Theorem}[section]
\newtheorem{lemma}[theorem]{Lemma}
\newtheorem{proposition}[theorem]{Proposition}
\newtheorem{corollary}[theorem]{Corollary}

\theoremstyle{definition}
\newtheorem{definition}[theorem]{Definition}

\theoremstyle{remark}
\newtheorem{remark}[theorem]{Remark}

\newcommand{\ba}{\begin{array}}
\newcommand{\ea}{\end{array}}

\begin{document}
\date{}
\title{ \bf\large{A Generalized Mountain Pass Lemma  with
a Closed Subset for Locally Lipschitz Functionals\\
 \begin{small} In memory of Professor Shi Shuzhong for his 80th birthday \end{small}}}
\author{Fengying Li\textsuperscript{1}\footnote{Corresponding Author, Email: lify0308@163.com}\ \
Bingyu Li\textsuperscript{2}\ \ Shiqing Zhang\textsuperscript{3}
 \\
{\small \textsuperscript{1} School of Economic and Mathematics,
Southwestern University of Finance and \hfill{\ }}\\
\ \ {\small Economics, Chengdu, Sichuan, 611130, P.R.China.\hfill{\ }}\\
{\small \textsuperscript{2} College of Information Science and Technology,
Chengdu University of Technology,\hfill{\ }}\\
\ \ {\small Chengdu, Sichuan, 610059, P.R.China.\hfill {\ }}\\
{\small \textsuperscript{3} Department of Mathematics,
 Sichuan University,\hfill{\ }}\\
\ \ {\small Chengdu, Sichuan, 610064, P.R.China.\hfill {\ }}}

\maketitle

\begin{abstract}
The classical Mountain Pass Lemma of Ambrosetti-Rabinowitz
has been studied,
extended and modified in several directions.
Notable examples would certainly
include the generalization to locally Lipschitz functionals by K.C. Chang,
analyzing the structure of the critical set in the mountain pass theorem
in the works of Hofer, Pucci-Serrin and Tian, and the extension
by Ghoussoub-Preiss
to closed subsets in a Banach space with recent variations. In this paper,
we utilize the generalized gradient of Clarke and Ekeland's
variatonal principle to generalize the
Ghoussoub-Preiss's Theorem in the setting of locally Lipschitz functionals.
We give an application to periodic solutions of Hamiltonian systems.

\noindent{{\bf Keywords}}: Mountain Pass Lemma of Ambrosetti-Rabinowitz,
Ekeland's variational principle,
locally Lipschitz functionals, Clarke's  generalized gradient,
generalized Mountain Pass Lemma.

\noindent{{\bf 2000 Mathematical Subject Classification}}: 34A34, 34C25, 35A15.
\end{abstract}

\section{Introduction and Main Results}
\renewcommand{\theequation}{\arabic{section}\arabic{equation}}

Saddle points in the Mountain pass Lemma ( \cite{1} - \cite{25}) are different from
maximum points and minimum points. Maximum and Minimum problems in infinite dimensional
space have a very long and prominent history ( \cite{22}) with "isoperimetric problems"
and the "problem of the brachistochrone" as two notable examples. In the 19th century
"Dirichlet principle" we essentially encountered the problem of minimizing a functional;
however, complete rigor was mostly lacking and we had to wait for Hilbert for satisfactory
completion of the Dirichlet principle. Continuing in the 20th century, Italian mathematician
Tonelli introduced the concept of a weakly lower semi-continuous (w.l.s.c) functional and
proved that a w.l.s.c functional defined on a weakly closed subset of a reflexive Banach
space can attain its infimum if it is coercive ( \cite{22}).
At times, the existence of a saddle point, which is neither a maximum nor minimum point,
is of considerable importance. Minimax methods in the finite dimensional case ( \cite{22}, \cite{25})
can be traced back to Birkhoff in 1917 and von Neumann's minimax theorem in 1928. We can also observe
that the Mountain Pass Lemma of Ambrosseti-Rabinowitz ( \cite{1}) in 1973 is a type of minimax theorem,
which can be traced back to Courant in 1950 for the finite dimensional case ( \cite{22}).
The Palais-Smale condition first appeared in connection with infinite dimensional problems,
but it is necessary even in finite dimensional situations. Reference ( \cite{4}) gives an example
of a polynomial in two variables that has exactly two non-degenerate critical points, both of
which are global minimizers, and points out that the given polynomial does not satisfy the
Palais-Smale condition. From the finite dimensional case to the infinite dimensional case,
the key step in the proof of the Mountain Pass Lemma is the use of a Palais-Smale type
compactness condition($PS$) to drive Palais's Deformation Lemma. We should note the original
proof of the Ambrosseti-Rabinowitz's Mountain Pass Lemma used Palais's Deformation Lemma ( \cite{22}).
In the 1970's, Ekeland discovered a very important principle for lower semi-continuous functions
on a complete metric space. Until the middle of the 1980's, Aubin-Ekeland ( \cite{3}), Shi ( \cite{21})
discovered the relationship between the Mountain Pass Lemma of Ambrosseti-Rabinowitz and Ekeland's
variational priciple. The Mountain Pass Lemma of Ambrosseti-Rabinowitiz has been intensively
studied and has found numerous applications ( \cite{1} - \cite{25}). Of special note,
it was generalized to the case of locally Lipschitz functionals by K.C. Chang ( \cite{5})
where he also obtained more minimax theorems by using a deformation lemma.

In this paper, we use Ekeland's variational principle to prove a generalized
Mountain Pass Lemma for locally Lipschitz functionals related with a closed subset,
and we also found an applications to Hamiltonian systems with local Lispschtiz potential and a fixed energy.

In 1973, Ambrosetti and Rabinowitz \cite{1} published the famous Mountain-Pass Theorem:
\begin{theorem}\label{thm 1}
{\rm (\cite{1})} Let $f$ be a $C^1-$real functional defined on a Banach space $X$ satisfying the following $(PS)$ condition:

Every sequence $\{x_n\}\subset X$ such that $\{f(x_n)\}$ is bounded and
 $\|f'(x_n)\|\rightarrow 0$ in $ X^*$ has a strongly convergent subsequence.

 Suppose there is an open neighborhood $\Omega$ of $x_0$ and a point $x_1\notin \bar{\Omega}$ such that
$
f(x_0),f(x_1)<c_0\leq\inf_{\partial\Omega}f,
$
and let
$$
\Gamma :=\{g\in C([0,1];X):g(0)=x_0, g(1)=x_1\}.
$$
Then
$
c :=\inf_{g\in\Gamma}\max_{t\in [0,1]}f(g(t))\geq c_0
$
 is a critical value of $f$: that is,
 there is $\bar{x}\in X$ such that $f(\bar{x})=c$ and $f'(\bar{x})=0$,
 where $f'(\bar{x})$ denotes the Fr\'{e}chet derivative of f at $\bar{x}$.
\end{theorem}

Let $C^{1-0}(X;\mathds{R})$ be the space of locally Lipschitz mappings from
$X$ to $\mathds{R}$. For $\Phi\in C^{1-0}(X;\mathds{R})$ set (Clarke\cite{6})
$$
\partial\Phi(x):=\{x^*\in X^*:<x^*,v>\leq\Phi^0(x,v),\forall v\in X\},
$$
where
$\Phi^0(x,v):=\limsup\limits_{\mathop{{w\rightarrow x} }\limits_{t\downarrow 0}}\frac{\Phi(w+tv)-\Phi(w)}{t} $
denotes the generalized directional derivative of $\Phi$ at the point $x$ along the direction $v$.
We should note that if $\Phi\in C^1$, then $\Phi^0(x,v)$ reduces to the G\^{a}teaux directional
derivative and $\partial\Phi$ reduces to the classical derivative.

\begin{proposition}
$\Phi^0(x,v)$ have some useful properties ( \cite{6}):
\begin{enumerate}
\item[(1)] The function $v\rightarrow \Phi^0(x,v)$ is subadditive and positively homogeneous, and then is convex,
\item[(2)] $|\Phi^0(x,v)|\leq K\|v\|$,
\item[(3)] The function $v \rightarrow \Phi^0(x,v)$ is continuous,
\item[(4)] $\Phi^0(x,-v)=(-\Phi)^0(x,v)$.
\end{enumerate}
\end{proposition}

K.C. Chang \cite{5} generalized the classical ($PS$) condition and the Mountain Pass Theorem
to local Lipschitz functions. Ribarska-Tsachev- Krastanov \cite{19} gave a generalization of
a result of Chang for the case when "the separating mountain range has zero altitude" which is
a version of the general mountain pass principle of Ghoussoub-Preiss for locally Lipschitz functions.

The generalization of the Mountain Pass Lemma of Ghoussoub-Preiss \cite{9} involves
the modification of the classical Palais-Smale condition:
\begin{definition}\label{def1}
Let $X$ be a Banach space, $F$ a closed subset of $X$ and $\varphi$ a G\^{a}teaux-differentiable
functional on $X$. The $(PS)_{F,c}$ condition is the following: if $\{x_n\}\subset X$ is a
sequence satisfying the three conditions
\begin{enumerate}
\item[(i)] $d(x_n,F)\rightarrow 0$,
where $d(x,F):=\inf\limits_{y\in F}||x-y||$ denotes the distance between the point $x$ and the set $F$,
\item[(ii)] $\varphi(x_n)\rightarrow c$,
\item[(iii)] $\varphi'(x_n)\rightarrow 0$,
\end{enumerate}
then $\{x_n\}$ has a strongly convergent subsequence.
\end{definition}

\begin{definition}\label{def2}
Let $X$ be a Banach space, and $F\subset X$ a closed subset. We say $\Phi\in C^{1-0}(X;\mathds{R})$
meets the $(CPS)_{F,c}$ condition when the following is true: if $\{x_n\}\subset X$ satisfies
\begin{enumerate}
\item[(1)] $d(x_n,F)\rightarrow0$,
\item[(2)] $\Phi(x_n)\rightarrow c$,
\item[(3)] $(1+\|x_n\|)\cdot\min\limits_{y^*\in \partial\Phi(x_n)}\|y^*\|\rightarrow 0$,
\end{enumerate}
then $\{x_n\}$ has a convergent subsequence.
\end{definition}

It should be noticed that in Definition \ref{def2}, the $(PSC)_{F,c}$ condition reduces to
the Cerami Condition when $F=X$ and $\Phi\in C^1(X;R)$.

We can define the $\delta$-distance (geodesic distance) ( \cite{8}):
\begin{equation}\label{eq2}
\delta(x_1,x_2):=\inf\{l(c): c\in C^1([0,1],X),c(0)=x_1,c(1)=x_2\},
\end{equation}
where
$$
l(c):=\int^1_0\frac{\|\dot{c}(t)\|}{1+\|c(t)\|}dt.
$$
we let $dist_{\delta}(x,F)=\inf\{\delta(x,y) : y\in F\}.$

\begin{proposition}
The geodesic distance $\delta$ has the following properties ( \cite{8}):
\begin{enumerate}
\item[(1).] $\delta (x_1,x_2)\leq \|x_1-x_2\|$,
\item[(2).] For every norm-bounded set $B$ in  $X$, there is $\gamma >0$
such that $\forall x_1, x_2\in B$, there holds $\delta(x_1,x_2)\geq\gamma\|x_1-x_2\|$.
\end{enumerate}
\end{proposition}

\begin{definition}\label{def3}
Let $X$ be a Banach space, and $F\subset X$ a closed subset. We say $\Phi\in C^{1-0}(X;\mathds{R})$
meets the $(CPS)_{F,c;\delta}$ condition when the following is true: if $\{x_n\}\subset X$ satisfies
\begin{enumerate}
\item[(1)] $dist_{\delta}(x_n,F)\rightarrow0$,
\item[(2)] $\Phi(x_n)\rightarrow c$,
\item[(3)] $(1+\|x_n\|)\cdot\min\limits_{y^*\in \partial\Phi(x_n)}\|y^*\|\rightarrow 0$,
\end{enumerate}
then $\{x_n\}$ has a convergent subsequence.
\end{definition}
We can now state the Mountain-Pass Theorem generalized by Ghoussoub-Preiss \cite{9}
for a continuous and G\^{a}teaux-differentiable functional statisfying the $(PS)_{F,c}$ condition:
\begin{theorem}{ \rm( \cite{9})}\label{thm 2}
Let $\varphi: X\rightarrow \mathds{R}$ be a continuous and G\^{a}teaux-differentiable
functional on a Banach space $X$ such that $\varphi':X\rightarrow X^*$ is continuous from
the norm topology of $X$ to the $w^*-$topology of $X^*$.
Take $u,v\in X$, and let
$$
c:=\inf_{g\in\Gamma}\max_{t\in [0,1]}\varphi(g(t))
$$
where $\Gamma=\Gamma_u^v$ is the set of all continuous paths joining $u$ and $v$.
Suppose $F$ is a closed subset of $X$ such that $F\cap\{x\in X:\varphi(x)\geq c\}$
separates $u$ and $v$ and $\varphi$ satisfies the $(PS)_{F,c}$ condition,
then there exists a critical point $\bar{x}\in F$ for $\varphi$ on $F$ with critical value $c$:
 $\varphi(\bar{x})=c,\varphi'(\bar{x})=0$.
\end{theorem}

A key ingredient in the proof of Theorem \ref{thm 2} is provided by the following
fundamental theorem in non-convex and nonlinear functional analysis established in
the 1974 paper of Ivar Ekeland \cite{7}.

\begin{theorem}{\rm( \cite{7})}\label{thm 4} Let $(X,d)$ be a complete metric space
with metric $d$ and $f: X\rightarrow \mathds{R}\cup\{+\infty\}$ a lower semi-continuous
functional not identically $+\infty$ which is bounded from below.
Let $\varepsilon>0$ and $u\in X$ such that
$
f(u)\leq\inf_{x\in X}f(x)+\varepsilon.
$
Then for any given $\lambda>0$, there exists $v_{\lambda}\in X$
such that $f(v_{\lambda})\leq f(u)$, $d(u,v_{\lambda})\leq\lambda$, and
$
f(w)>f(v_{\lambda})-\frac{\varepsilon}{\lambda} d(v_{\lambda},w),\ \ \forall w\neq v_{\lambda}.
$
\end{theorem}

Ekeland's variational principle has found numerous applications;
in particular, we would like to observe that prior to Ghoussoub-Preiss \cite{9}
it was used by Shi \cite{21} to prove a Mountain Pass Lemma and general min-max theorems for
locally Lipschitz functionals (K.C.Chang \cite{5}). In this paper, we will use Ekeland's
variational principle to generalize the Ghoussoub-Preiss Theorem to the case of locally Lipschitz functional of
class $C^{1-0}$ satisfying the conditions $(CPS)_{F,c;\delta}$ or $(CPS)_{F,c}$.

\begin{theorem}\label{thm 5} Let $X$ be a Banach space with norm $||.||$, $C^{0}([0,1];X)$
the space of continuous mappings from $[0,1]$ to $X$, and $\Phi: X\rightarrow \mathds{R}$ a locally Lipschitz functional.
 For $z_0,z_1\in X$, define
$$
\Gamma :=\{c\in C^0([0,1];X):c(0)=z_0,c(1)=z_1\},
\gamma :=\inf\limits_{c\in \Gamma}\max\limits_{0\leq t\leq 1}\Phi(c(t)),
$$
and set
$
\Phi_{\gamma}:=\{x\in X:\Phi(x)\geq \gamma\}.
$
If $F\subset X$ is a closed subset such that $F\cap\Phi_{\gamma}$ separates $z_0$ and $z_1$,
then there exists a sequence $\{x_n\}\subset X$ such that
$
dist_{\delta}(x_n,F)\rightarrow 0$, $\Phi(x_n)\rightarrow \gamma$ and
$(1+\|x_n\|)\min\limits_{y^*\in \partial\Phi(x_n)}\|y^*\|\rightarrow 0$.
\end{theorem}

\begin{theorem}\label{thm 6}
Under the assumptions of Theorem \ref{thm 5}, if we add that the set $F$ is
norm-bounded in the Banach space $X$, then there exists a sequence
$\{x_n\}\subset X$ such that
$
d(x_n,F)\rightarrow 0$, $\Phi(x_n)\rightarrow \gamma$ and
$(1+\|x_n\|)\min\limits_{y^*\in \partial\Phi(x_n)}\|y^*\|\rightarrow 0$.
\end{theorem}

\begin{theorem}\label{thm 7}
Under the assumptions of Theorem \ref{thm 5}, if $\Phi$ satisfies $(CPS)_{F,\gamma;\delta}$ condition,
then $\gamma$ is a critical value for $\Phi$:
$\Phi(\bar{x})=\gamma,\ \ 0\in\Phi'(\bar{x})$.
\end{theorem}
\begin{theorem}\label{thm 8}
Under the assumptions of Theorem \ref{thm 5}, if we add the condition that the set $F$ is
bounded in the norm of the Banach space $X$, then we can change the
$(CPS)_{F,\gamma;\delta}$ condition to the $(CPS)_{F,\gamma}$ condition,
and conclude there exists a critical point $\bar{x}\in F$ for
$\Phi$ on $F$ with critical value $\gamma: \Phi(\bar{x})=\gamma,0\in\Phi'(\bar{x})$.
\end{theorem}

In 2009, Goga \cite{10} studied a general Mountain Pass Theorem for local Lipschitz function.
Let $(E,\|\cdot\|)$ be a Banach space, $S$ a compact metric space and $S_0$ a closed subset of $S$.
Let $C(S,E)$ be the Banach space of all $E-$valued bounded continuous mapping on $S$
with the norm $\|\gamma\|:=\sup\limits_{x\in S} \|\gamma(x)\|$. Let $\gamma_0\in C(E,S)$
be a fixed element and define $$\Gamma=\{\gamma\in C(S,E):\gamma(s)=\gamma_0(s),\forall s\in S_0\},
c:=\inf_{\gamma\in \Gamma}\sup_{s\in S}f(\gamma(s)),$$ where $f$ is a real-valued function defined on $E$.
Goga's result is the following:
\begin{theorem}\label{thm 3}
Let $f: E\rightarrow \mathds{R}$ be a locally Lipschitz function and $F$ a closed nonempty subset of $E$.
Assume that
\begin{enumerate}
\item[(a)] $\gamma(S)\cap F\cap f_c\neq \emptyset, \forall \gamma\in \Gamma$, where $f_c=\{x\in E:f(x)\geq c\}$,
\item[(b)] $dist(\gamma_0(S_0),F)>0$, where $dist(\cdot,F)$ is the distance function to $F$ in $E$.
\end{enumerate}

Then for every $\varepsilon>0$ there exist $x_\varepsilon\in E$ such that
\begin{enumerate}
\item[(i)] $dist(x_\varepsilon,F)<\frac{3\varepsilon}{2}$,
\item[(ii)] $c\leq f(x_\varepsilon)<\varepsilon+\frac{5\varepsilon^2}{4}$,
\item[(iii)] $dist(0, \partial f(x_\varepsilon))\leq2\varepsilon$, where $\partial f(x)$ is the
Clark sub-differential of $f$ at $x$.
\end{enumerate}
\end{theorem}

Thanks to a referee for pointing out the papers \cite{Liv} - \cite{Liv2}.
In \cite{Liv2}, Livrea, R. and Marano S. A. considered more general compactness conditions: Let
$h:[0,+\infty[\rightarrow  [0,+\infty[ $ be a continuous function enjoying the following property:
\begin{equation}\label{lth}
\int_0^{+\infty}\frac{1}{1+h(\xi)}d\xi=+\infty.
\end{equation}
 They call that $f$ satisfies a weak Palais-Smale condition at the level $c\in \mathbb{R}$
 when for some $h$ as above one has:\\
 $(PS)_c^h$
\emph{Every sequence $\{x_n\}\subseteq X$ such that
\begin{equation}\label{ps}
\lim_{n\rightarrow +\infty}f(x_n)=c \ \  \text{and} \ \
\lim_{n\rightarrow +\infty}(1+h(\|x_n\|))\min_{y^\ast\in \partial f(x_n)}\|y^\ast\|=0
\end{equation}
possesses a convergent subsequence.}

A weaker form of $(PS)_c^h$ is the one below, where $U$ denotes a nonempty closed subset of $X$.
For $U: =X$ it coincides with $(PS)_c^h$.\\
$(PS)_{U,c}^h$ \emph{Every sequence $\{x_n\}\subseteq X$ such that $d(x_n,U)\rightarrow 0$
as $n\rightarrow+\infty$ and (\ref{ps}) holds true possesses a convergent subsequence.}

Given $x,z\in X$, denoted by $\mathcal{P}(x,z)$ the family of all piecewise
$C^1$ paths $p:[0,1]\rightarrow X$ such that $p(0)=x$ and $p(1)=z$. Moreover, put
$$
l_h(p):=\int_0^1\frac{\|p'(t)\|}{1+h(\|p(t)\|)}dt,\ \ p\in\mathcal{P}(x,z),
$$
as well as
\begin{equation}\label{dist}
\delta_h(x,z):=\inf \{l_h(p): p\in \mathcal{P}(x,z) \}.
\end{equation}

Let $B$ be a nonempty closed subset of $X$ and let $\mathcal{F}$
be a class of nonempty compact sets in $X$.
According to \cite{GH} Definition 1, they call that $\mathcal{F}$ is a
\emph{homotopy-stable family with extended boundary} $B$ when for every
$A\in\mathcal{F}, \eta\in C^0([0,1]\times X,X)$ such that $\eta(t,x)=x$
on $(\{0\}\times X)\cup([0,1] \times B)$ one has $\eta(\{1\}\times A) \in \mathcal{F}$.
Some meaningful situations are special cases of this notion. For instance,
if $Q$ denotes a compact set in $X$, $Q_0$  is a non-empty closed subset of $Q$,
$\gamma_0\in C^0 (Q_0,X)$,
$$\Gamma:= \{\gamma \in C^0(Q,X):\gamma|Q_0 =\gamma_0\},$$
and $\mathcal{F}:=\{\gamma(Q):\gamma\in \Gamma\}$, then $\mathcal{F}$ enjoys the above-mentioned property with
$B:= \gamma_0(Q_0)$.  In particular, it holds true when $Q$ indicates a compact topological
manifold in $X$ having a nonempty boundary $Q_0$  while $\gamma_0:=id|Q_0$.

In \cite{Liv2}, Livrea and Marano made the following assumptions:\\
\emph{($a_1$) $f : X \rightarrow \mathds{R}$ is a locally Lipschitz continuous function.\\
($a_2$) $\mathcal{F}$ denotes a homotopy-stable family with extended boundary $B$.\\
($a_3$)  There exists a nonempty closed subset $F$ of $X$ such that
$$(A \cap F)\backslash B \neq \emptyset, \ \ \forall A \in \mathcal{F} $$
and, moreover,
$$\sup_{x\in B}f(x)\leq \inf_{x\in F}f(x).$$
($a_4$) $h : [0 , +\infty[\rightarrow [0, +\infty[ $ is a continuous function fulfilling (\ref{ps}),
while $\delta_h$ indicates the metric defined (\ref{lth}).\\
Set, as usual,
$$ c:=\inf_{A\in\mathcal{F}}\max_{x\in A}f(x).$$ }

Livrea and Marano \cite{Liv}-\cite{Liv2} obtained the following theorems
\begin{theorem}\label{m1}
Let ($a_1$)-($a_4$) be satisfied. Then to every sequence $\{A_n \}\subseteq \mathcal{F}$
such that $\lim\limits_{n\rightarrow +\infty}\max\limits_{ x\in A_n}f(x)=c$
there corresponds a sequence ${x_n} \subseteq X\backslash B$ having the following properties:\\
($i_1$) $\lim\limits_{n\rightarrow +\infty} f(x_n) = c $.\\
($i_2$) $ (1 + h(\| x_n\|))f^0 (x_n;z)\geq -\epsilon_n\|z\|$   for all $n\in \mathds{N}, z\in X$,
where $\epsilon_n \rightarrow 0^+ $.\\
($i_3$) $\lim\limits_{n\rightarrow +\infty} \delta_h(x_n,F) = 0$ provided $\inf_{x\in F} f(x) = c $.\\
($i_4$) $\lim\limits_{n\rightarrow+\infty }\delta_h(x_n,A_n ) = 0 $.
\end{theorem}
\begin{theorem}\label{m2}
 Let ($a_1$) - ($a_4$) be satisfied. Suppose that either $(PS)_c^h$ holds or
$F$ is bounded and $(PS)_{F,c}^h$ holds, according to whether
$\inf\limits_{x\in F} f(x)< c$ or $\inf\limits_{x\in F }f(x) = c$.
Then $K_c(f)=\{x\in X: f(x)=c, \text{ x is a critical point of f } \}\neq \emptyset$.
If, moreover, $\inf\limits_{x\in F}f(x)=c$, then $K_c(f)\cap F\neq \emptyset $.
\end{theorem}
\begin{theorem}\label{m3}
Let ($a_1$) and ($a_4$) be satisfied. Suppose that:\\
($a_5$) There exists a closed subset $F$ of $X$ complying with
$(\gamma(Q)\cap F)\backslash \gamma_0(Q_0)\neq\emptyset $
for all $ \gamma\in\Gamma$ and, moreover,
$\sup\limits_{x\in Q_0} f(\gamma_0(x))\leq \inf\limits_{x\in F}f(x)$.\\
($a_6$)  Setting $c:=\inf\limits_{\gamma\in\Gamma}\max\limits_{x\in Q} f(\gamma(x))$,
either $(PS)_c^h$ holds or $F$ is  bounded and $(PS)_{F,c}^h$ holds, according to whether
$\inf\limits_{x\in F} f(x) < c$ or $\inf\limits_{x\in F} f(x) = c$.
Then the conclusion of Theorem (\ref{m2}) is true.
\end{theorem}

\begin{remark}
Comparing our Theorem \ref{thm 5} with Livrea - Marano's
Theorem \ref{m1}, we found that although the conditions of
Livrea - Marano's theorems are more general, but their results are weaker,
in fact, in Theorem \ref{m1} they can obtain
$\lim\limits_{n\rightarrow+\infty}\delta_h(x_n,F)=0$
under the assumptions $\inf\limits_{x\in F}f(x)=c$,
without this condition, they can't get\\
$\lim\limits_{n\rightarrow +\infty}\delta_h(x_n,F)=0$,
generally, they can only obtain ($i_4$)
$\lim\limits_{n\rightarrow+\infty}\delta_h(x_n,A_n)=0$.
In Theorem \ref{m2} and Theorem \ref{m3}, although the
conditions ($a_2$)-($a_4$) are more general than our
Theorems \ref{thm 6}-\ref{thm 8}, but they suppose two cases that
"either $(PS)_c^h$ holds or $F$ is bounded and  $(PS)_{F,c}^h$ holds,
according to whether
$\inf\limits_{x\in F}f(x)<c$ or $\inf\limits_{x\in F}f(x)=c$",
which are more difficulty to apply, in fact, in many applied examples,
it's very difficult to prove
$\inf\limits_{x\in F}f(x)<c$ or $\inf\limits_{x\in F}f(x)=c$.
\end{remark}

\begin{remark}
It should be poited out that the conditions $(CPS)_{F,c}$ and
$(CPS)_{F,c,\delta}$ generalize
the Cerami condition (with a closed subset $F$) in the smooth situation.
For more on Cerami sequence
in the smooth case, we refer to Schechter \cite{20} and Stuart \cite{23},
Stuart's paper proved the
smooth case of our Theorem \ref{thm 5} and Theorem \ref{thm 6}.
The conclusions (i)-(iii) of Goga's Theorem \ref{thm 3} and the condition $(PS)_c$ in
 Ribarska- Tsachev- Krastanov \cite{19} are different from
the conditions $(CPS)_{F,c}$ and $(CPS)_{F,c;\delta}$
stated here. Our results Theorem \ref{thm 5} and Theorem \ref{thm 6} are stronger since
 $(1+\|x_n\|)\min\limits_{y\ast\in \partial \Phi(x_n)}\|y\ast\|\rightarrow 0$
implies (iii) of the Theorem \ref{thm 3}.
 We would also like to note the assumptions in Theorem \ref{thm 7} and Theorem \ref{thm 8},
and our $(CPS)_{F,\gamma;\delta}$ and $(CPS)_{F,\gamma}$
conditions are weaker than those used \cite{10} and \cite{19};
therefore, the arguments in our paper differ from \cite{10} and \cite{19}
since they could utilize the Borwein-Preiss
variational principle or a deformation lemma, whereas we use
the classical Ekeland's variational principle.
\end{remark}

\begin{remark}
We should note the difference between our Generalized Mountain Pass Lemma (GMPL) and the following
theorem of Struwe( \cite{22}): Suppose $M$ is a closed convex subset of a Banach space
$V$ and $E\in C^1(V)$ satisfies $(P.-S.)_M$ on $M$.
 Any sequence $\{u_n\}\subset M$ such that $|E(u_n)|\leq c$ uniformly, while
 $g(u_m)=\sup\limits_{\mathop{{v\in M} }\limits_{\|u_m-v\|<1}}\langle u_m-v,DE(u_m)
\rangle\rightarrow0$
$(m\rightarrow\infty)$,
is relatively compact.
Suppose further that $E$ admits two distinct relative minima $u_1$, $u_2$ in $M$.
Then either $E(u_1)=E(u_2)=\beta$ and $u_1$, $u_2$
can be connected in any neighborhood of the set of relative minima $u\in M$ of
$E$ with $E(u)=\beta$,
or there exists a critical point $\bar{u}$ of $E$ in $M$ which is not a relative minimizer of $E$.

In Struwe's Theorem, $M$ is a closed convex subset of a Banach space, but in our
GMPL we don't assume any convexity.
We also don't assume that $E: M\rightarrow \mathds{R}$ possesses an
extension $E\in C^1(V;\mathbb{R})$ to $V$,
but only that the functional is locally Lipschitz. Struwe's Theorem assumes
the existence of two local minimizers,
but we only require the existence of two valleys which may not be local minimizers. In these ways,
we see the premise in our GMPL is weaker than the corresponding conditions in Struwe's Theorem.
\end{remark}

\begin{remark}
The classical Mountain Pass Lemma and its many generalizations
are primarily concerned with saddle points, but we should note the saddle
points encountered in these various Mountain Pass Lemmas are different from
those in the Von Neumann Minimax Theorem ( \cite{25}). The Minimax Theorem
of Neumann is essentially related with convexity and concavity, whereas the
Mountain Pass Lemma is not related with convexity and concavity which is
related with ($PS$) compactness condition and two valleys for functional. It
seems interesting to use Ekeland’s variational principle to prove von Neumann Minimax Theorem.
\end{remark}

\section{The Proofs of Theorems \ref{thm 5}-\ref{thm 8}}

\begin{proof}
Since the main ingredient is still Ekeland's variational principle, we utilize some notations and
ideas from \cite{8} and \cite{9}, but must deviate in a few key steps.
Since the closed set $F_{\gamma}:=\Phi_{\gamma}\bigcap F$ separates $z_0$ and $z_1$, we can write
$
X\setminus F_{\gamma}:=\Omega_0\bigcup\Omega_1
$
where $z_0\in \Omega_0$, $z_1\in \Omega_1$ for open sets $\Omega_0$
and $\Omega_1$ with $\Omega_0 \cap \Omega_1=\emptyset$.

Choose $\varepsilon$ which satisfies
\begin{equation}\label{eq5}
0<\varepsilon<\frac{1}{2}\min\{1,dist_{\delta}(z_0,F_{\gamma}),dist_{\delta}(z_1,F_{\gamma})\}.
\end{equation}
By the definition of $\Gamma$, we can find $c\in\Gamma$ such that
\begin{equation}\label{eq6}
\max_{0\leq t\leq1}\Phi(c(t))<\gamma+\frac{\varepsilon^2}{4}.
\end{equation}
If we define $t_0$ and $t_1$ by
\begin{eqnarray*}
t_0&:=&\sup\{t\in[0,1]:c(t)\in\Omega_0,dist_{\delta}(c(t),F_{\gamma})\geq\varepsilon\},\\
t_1&:=&\inf\{t\in[t_0,1]:c(t)\in\Omega_1,dist_{\delta}(c(t),F_{\gamma})\geq\varepsilon\},
\end{eqnarray*}
then since $c(0)=z_0$, we have by (\ref{eq5}) and the continuity of $c$ that
$t_0>0$; moreover, by $c(t_0)\in\bar{\Omega}_{0}$ and
$dist_{\delta}(c(t_0),F_{\gamma})\geq\varepsilon$, we have
$c(t_0)\in\Omega_0$. Then $\Omega_0\cap\Omega_1=\emptyset$ implies
$t_1>t_0$. Again by (\ref{eq5}) and the continuity of $c$ we have $t_1<1$. So altogether
$0<t_0<t_1<1$.
Let
\begin{equation}\label{eq7}
\Gamma(t_0,t_1):=\{f\in C^0([t_0,t_1],X):f(t_0)=c(t_0),f(t_1)=c(t_1)\},
\end{equation}
and consider the following distance in $\Gamma(t_0,t_1)$:
\begin{equation}\label{eq8}
\rho(f_1,f_2):=\max_{t_0\leq t\leq t_1}\delta(f_1(t),f_2(t)),
\end{equation}
where
\begin{equation}\label{eq9}
\delta(x_1,x_2):=\inf\{l(c) : c\in C^1([0,1],X),c(0)=x_1,c(1)=x_2\}
\end{equation}
with
\begin{equation}\label{eq10}
l(c):=\int^1_0\frac{\|\dot{c}(t)\|}{1+\|c(t)\|}dt.
\end{equation}
For $x\in X$, we define the function
\begin{equation}\label{eq11}
\Psi(x):=\max\{0,\varepsilon^2-\varepsilon dist_{\delta}(x,F_{\gamma})\}.
\end{equation}
 A map $\varphi: \Gamma(t_0,t_1)\rightarrow \mathds{R}$ is defined by
 \begin{equation}\label{eq12}
\varphi(f):=\max_{t_0\leq t\leq t_1}\{\Phi(f(t))+\Psi(f(t))\}.
\end{equation}
Since $f(t_0)=c(t_0)\in\Omega_0, f(t_1)=c(t_1)\in\Omega_1$,
there exists $t_f\in (t_0,t_1)$ satisfying $f(t_f)\in \partial\Omega_0\subset F_{\gamma}$; therefore,
\begin{equation}\label{eq14}
dist_{\delta}(f(t_f),F_{\gamma})=0,
\end{equation}
and for any $f\in\Gamma(t_0,t_1)$, we have
\begin{equation}\label{eq15}
\varphi(f)\geq\Phi(f(t_f))+\Psi(f(t_f))\geq\gamma+\varepsilon^2.
\end{equation}
On the other hand, if we denote $\hat{c}=c|_{[t_0,t_1]}$, then
\begin{equation}\label{eq16}
\varphi(\hat{c})\leq \max_{0\leq t\leq 1}\{\Phi(c(t))+\Psi(c(t))\}\leq\gamma+\frac{5}{4}\varepsilon^2.
\end{equation}
Notice that $\Gamma(t_0,t_1)$ is a complete metric space with respect to the metric $\rho$
introduced in (\ref{eq8}) \cite{7},\cite{8}.
Since $\Phi$ and $\Psi$ are lower semi-continuous, so is $\varphi$. Now (\ref{eq15}) implies $\varphi$
has a lower bound, and by (\ref{eq15}) and (\ref{eq16}) we have
\begin{equation}\label{eq17}
\varphi(\hat{c})\leq\inf \varphi+\frac{\varepsilon^2}{4}.
\end{equation}
In Ekeland's variational principle, we use
$\frac{\varepsilon^2}{4}$ in place of $\varepsilon$, and take
$\lambda=\frac{\varepsilon}{2}$, then there exists $\hat{f}\in
\Gamma(t_0,t_1)$ such that
\begin{equation}\label{eq18}
\begin{array}{l}
\varphi(\hat{f})\leq\varphi(\hat{c}),\ \
\rho(\hat{f},\hat{c})\leq\frac{\varepsilon}{2},\ \
\varphi(f)\geq\varphi(\hat{f})-\frac{\varepsilon}{2}\rho(f,\hat{f}), \ \ \forall f\in \Gamma(t_0,t_1).
\end{array}
\end{equation}
Let
\begin{equation}\label{eq19}
M:=\{t\in [t_0,t_1]:\Phi(\hat{f}(t))+\Psi(\hat{f}(t))=\varphi(\hat{f})\}.
\end{equation}

Note that $M$ is a non- empty set that is compact, since $\Phi$ and $\Psi$ are lower- semicontinuous.
We next show that $t_0,t_1\notin M$.

By the definitions of $t_0$ and $t_1$, we have
\begin{equation}\label{eq20}
dist_{\delta}(c(t_i),F_{\gamma})\geq \varepsilon, \ \ i=0,1,
\end{equation}
so by (\ref{eq11})we have $\Psi(\hat{c}(t_i))=0$. By (\ref{eq6}) and (\ref{eq15}) and (\ref{eq18} ) we have
\begin{equation}\label{eq21}
\Phi(\hat{f}(t_i))+\Psi(\hat{f}(t_i))\leq\Phi(\hat{c}(t_i))+\Psi(\hat{c}(t_i))\leq \gamma
+\frac{\varepsilon^2}{4}<\varphi(\hat{f}),i=0,1,
\end{equation}
which implies $t_0,t_1\notin M$.
\vspace{0.2cm}

\textbf{\underline{Claim}}: There exists $t\in M $ such that
\begin{equation}\label{eq22}
\min_{x^*\in \partial\Phi(\hat{f}(t))}\|x^*\|(1+\|\hat{f}(t)\|)\leq \frac{3\varepsilon}{2}.
\end{equation}

\begin{proof} If not, for any $t\in M$,
\begin{equation}\label{eq23}
\min_{x^*\in \partial\Phi(\hat{f}(t))}\|x^*\|(1+\|\hat{f}(t)\|)>\frac{3\varepsilon}{2}.
\end{equation}
It is well known that $\|x^*\|=\sup_{v\neq 0}\frac{<x^*,v>}{\|v\|}$ where
\begin{equation}\label{eq24}
x^*\in \partial\Phi(\hat{f}(t))=\{x^*\in X^*:<x^*,v>\leq\Phi^0(\hat{f}(t),v),\forall v\in X\},
\end{equation}
where the following definition
$$\Phi^0(x,v)=\limsup_{\mathop{{w\rightarrow x} }\limits_{t\downarrow 0}}\frac{\Phi(w+tv)-\Phi(w)}{t} $$
denotes the generalized directional derivative of $\Phi$ at the point $x$ along the direction $v$.

Let $u\in \partial\Phi(\hat{f})$ satisfy $\|u\|=\min\limits_{x^*\in \partial\Phi(\hat{f})}\|x^*\|$.
By properties of Clarke's generalized gradient, we know that $\langle u,v\rangle \leq \Phi^0(\hat{f},v) $
for any $v\in X$ (from the reference \cite{5}).
Since $\frac{1}{1+\|\hat{f}\|}\frac{3\varepsilon}{2}<\min\limits_{x^*\in \partial\Phi(\hat{f})}\|x^*\|$,
we see that $\frac{1}{1+\|\hat{f}\|}\frac{3\varepsilon}{2}<\|u\|$,  and so there must be a $v\in X$
such that $\|v\|=1$ and  $\frac{1}{1+\|\hat{f}\|}\frac{3\varepsilon}{2}<\langle u, v\rangle$.
Setting $w:=(1+\|\hat{f}\|)v$, we see $\frac{3\varepsilon}{2}< \langle u,w\rangle$. Moreover,
since $\langle u,v\rangle\leq \Phi^0(\hat{f},v)$ for any $v\in X$,
we will have $\frac{3\varepsilon}{2}<\Phi^0(\hat{f},w)$.
 Finally, let $y := -w$. Then $\|y\|=\|w\|=1+\|\hat{f}\|$, and by properties of Clarke's
generalized gradient (see (4) on page 104 of reference \cite{5}) we have
$$\frac{3\varepsilon}{2}<\Phi^0(\hat{f},w)=-\Phi^0(\hat{f},-w)=-\Phi^0(\hat{f},y).$$
Notice that the left side of the above inequality is equal to
$\Phi^0(\hat{f}(t),h)$, so we get $\Phi^0(\hat{f}(t),u(t))<-\frac{3\varepsilon}{2}$ for $u(t)=h$ and $\|u\|=1+\|\hat{f}(t)\|$. Let
$N(t):=\{s\in M:\Phi^0(\hat{f}(s),u(t))<-\frac{3}{2}\varepsilon\}.$
Since $x\mapsto\Phi^0(x,v)$ is upper semicontinuous for any given $v$, $N(t)$ is an open subset of $M$ and since $t\in N(t)$,
$M$ can be covered by $\{N(t):t\in M\}.$
Since $M$ is compact, we can pick a finite open sub-cover for $M$, $\left\{ N(t_k):0\leq k\leq K \right\}$.
Then for the partition of unity associated with this cover on $M$, there are continuous
functions $\xi_k(t):0\leq\xi_k(t)\leq 1$ for $0\leq k\leq K$ with $\sum_0^K\xi_k(t)=1.$

Let $v(t):=\sum_{k=0}^K\xi_k(t)u(t)$ and from the sub-additivity and positive homogenity of $\Phi^{0}$
in its second argument, we observe the continuous map $v: M\rightarrow X$ satisfies
\begin{equation}\label{eq27}
\Phi^0(\hat{f}(t),v(t))<-\frac{3\varepsilon}{2}, \ \
\|v(t)\|\leq\sum_{k=0}^K\xi_k(t)(1+\|\hat{f}(t)\|=(1+\|\hat{f}(t)\|).
\end{equation}
Since $M\subset[t_0,t_1]$ and $M$ is a nonempty compact set with
$t_0,t_1\notin M$, so by Tietze extension theorem, we can extend $v$
to a continuous function defined on $[t_0,t_1]$ (which we still
denote by $v$) which satisfies $v(t_0)=v(t_1)=0$ and
\begin{equation}\label{eq28}
 \|v(t)\|\leq(1+\|\hat{f}(t)\|), \forall t\in[t_0,t_1].
\end{equation}
Since $v(t_0)=v(t_1)=0$, $\forall h>0$, $\hat{f}+hv\in\Gamma (t_0,t_1)$; hence by ( \ref{eq18}), we have
\begin{equation}\label{eq29}
 \varphi(\hat{f}+hv)\geq \varphi(\hat{f})-\frac{\varepsilon}{2}\rho(\hat{f}+hv,\hat{f}).
\end{equation}
 By (\ref{eq12}), we can choose $t_h\in[t_0,t_1]$ such that $\varphi (f)=\Phi(f(t_n))+\Psi(f(t_n))$ and then:
\begin{equation}\label{eq30}
 \varphi(\hat{f}+hv)=(\Phi+\Psi)(\hat{f}(t_h)+hv(t_h)).
\end{equation}
Notice that here $t_h$ is defined for each $h>0$. By the definition
of $\varphi$, we know that for any $h>0$, there holds $\varphi(\hat{f})\geq (\Phi+\Psi)(\hat{f}(t_h))\label{eq31}$.
 Combining with (\ref{eq29}) and (\ref{eq30}), we have for every $h > 0$:
\begin{equation}\label{eq32}
(\Phi+\Psi)(\hat{f}(t_h)+hv(t_h))\geq(\Phi+\Psi)(\hat{f}(t_h))-\frac{\varepsilon}{2}\rho(\hat{f}+hv,\hat{f});
\end{equation}
that is,
\begin{equation}\label{eq33}
\Phi(\hat{f}(t_h)+hv(t_h))-\Phi(\hat{f}(t_h))\geq-\Psi(\hat{f}(t_h)+hv(t_h))+\Psi(\hat{f}(t_h))-
\frac{\varepsilon}{2}\rho(\hat{f}+hv,\hat{f})).
\end{equation}
If we recall the definition of $\Psi$, then $\Psi$ is $\varepsilon-$Lipschitz with respect to the metric $\rho$,
and so the above inequality implies
\begin{equation}\label{eq34}
\Phi(\hat{f}(t_h)+hv(t_h))-\Phi(\hat{f}(t_h))\geq-\frac{3\varepsilon}{2}\rho(\hat{f}+hv,\hat{f})).
\end{equation}
Notice that if $h_n\rightarrow 0^+$, we can pass to a sequence $\{t_{h_n}\}$
with $t_{h_n}\rightarrow\tau\in M$ since $M$ is compact. Calculating
\begin{equation}\label{eq35}
\limsup_{n\rightarrow +\infty}\frac{\Phi(\hat{f}(t_{h_n})+h_nv(t_{h_n}))-\Phi(\hat{f}(t_{h_n}))}{h_n}
\geq-\frac{3\varepsilon}{2}\liminf_{n\rightarrow +\infty}\frac{\rho(\hat{f}+h_nv,\hat{f}))}{h_n}
\end{equation}
and further by $\Phi\in C^{1-0}$ and the definitions of Clark's generalized gradient and the metric $\rho$,
 we claim
\begin{equation}\label{eq36}
\Phi^0(\hat{f}(\tau),v(\tau))\geq-\frac{3\varepsilon}{2}\max_{t_0\leq t\leq t_1}
(\frac{\|v(t)\|}{1+\|\hat{f}(t)\|})\geq-\frac{3\varepsilon}{2}.
\end{equation}
In fact, by $\Phi\in C^{1-0}$ and the continuity for $v(t)$, we know that
$$\frac{\Phi(\hat{f}(t_{h_n})+h_nv(t_{h_n}))-
\Phi(\hat{f}(t_{h_n})+h_nv(\tau))}{h_n}\leq L|v(t_{h_n})-v(\tau)|\rightarrow 0,$$
hence
\begin{eqnarray*}
&\lim\sup\limits_{n\rightarrow +\infty}&\frac{\Phi(\hat{f}(t_{h_n})+h_nv(t_{h_n}))
-\Phi(\hat{f}(t_{h_n}))}{h_n}\\
&\leq& \lim\sup_{n\rightarrow +\infty}\frac{\Phi(\hat{f}(t_{h_n})
+h_nv(t_{h_n}))-\Phi(\hat{f}(t_{h_n})+h_nv(\tau))}{h_n}\\
&&+\ \ \lim\sup_{n\rightarrow +\infty}\frac{\Phi(\hat{f}(t_{h_n})+h_nv(\tau))
-\Phi(\hat{f}(t_{h_n})}{h_n}\\
&=&\Phi^0(\hat{f}(\tau),v(\tau)).
\end{eqnarray*}
Using the definition (\ref{eq8}) of the metric $\rho$, we have
\begin{eqnarray*}
&&\rho(\hat{f}+h_nv,\hat{f})\\
&&=\max_{t_0\leq t\leq t_1}\inf\{\int_0^1\frac{||\dot{c}(s)||}{1+||c(s)||}ds,\\
&&\ \ \ \ c(s)\in C^1([0,1],X),c(0)=\hat{f}(t),c(1)=(\hat{f}+h_nv)(t)\}.
\end{eqnarray*}
Specifically, if we take the following loop connecting $\hat{f}$ and
$\hat{f}+h_nv$ for $0\leq s \leq 1$:
$$c_{t,n}(s)=(1-s)\hat{f}(t)+s(\hat{f}(t)+h_nv(t))=\hat{f}(t)+sh_nv(t),$$
then we have $c_{t,n}(s)=h_nv(t)$, and so we have the uniform convergence in $t$ and $s$,
$$c_{t,n}(s)=\hat{f}(t)+sh_nv(t)\rightarrow \hat{f}(t),n\rightarrow +\infty.$$
So we have that $\int_0^1\frac{1}{1+\|c_{t,n}(s)\|}ds\rightarrow
\int_0^1\frac{1}{1+\|\hat{f}(t)\|}ds=\frac{1}{1+\|\hat{f}(t)\|}$,
$$\liminf_{n\rightarrow +\infty}\frac{\rho(\hat{f}+h_nv,\hat{f}))}{h_n}\leq
\max_{t_0\leq t\leq t_1}(\frac{\|v(t)\|}{1+\|\hat{f}(t)\|}),$$
and (\ref{eq36}) is proved, which violates (\ref{eq27}) and shows that we cannot have the inequality
(\ref{eq23}) for every $t\in M$. Therefore, there is $\bar{t}\in M$ such that
\begin{equation}\label{eq37}
\min_{x^*\in \partial\Phi(\hat{f}(\bar{t}))}\|x^*\|(1+\|\hat{f}(\bar{t})\|)\leq\frac{3\varepsilon}{2},
\end{equation}
this ends the proof of claim (\ref{eq22}).
\end{proof}

By the definitions of $t_0$ and $t_1$, we have that
$dist_\delta(\hat{c}(t),F_{\gamma})\leq\varepsilon$ for
$t_0<t<t_1$; furthermore, by continuity of $\hat{c}(t)$
and $dist_\delta(x,F_{\gamma})$ on $x$, we have that
$$dist_\delta(\hat{c}(t),F_{\gamma})\leq\varepsilon,\forall t\in[t_0,t_1].$$
Here, we have used the notation $dist_{\delta}(\hat{c}(t),F_{\gamma})$
to denote the distance between $\hat{c}(t)$ and $F_{\gamma}$ deduced
by $\delta$ in (\ref{eq9}).
We notice that $\rho$ is the distance deduced by $\delta$ in (\ref{eq9}),
since $\rho(\hat{f},\hat{c})\leq \frac{\varepsilon}{2}$, the triangle inequality implies
that for all $t\in [t_0,t_1]$ we have
\begin{equation}\label{eq38}
dist_{\delta}(\hat{f}(t),F_{\gamma})\leq\frac{\varepsilon}{2}+dist_{\delta}(\hat{c}(t),F_{\gamma}))\leq \frac{\varepsilon}{2}+\varepsilon=\frac{3\varepsilon}{2}.
\end{equation}
Set $x=\hat{f}(\bar{t})$, we get $dist_{\delta}(x,F_{\gamma})\leq\frac{3\varepsilon}{2}.$

Then by (\ref{eq18}), $\varphi(\hat{f})\leq \varphi(\hat{c})$; the fact that $\bar{t}\in M$,
(\ref{eq19}), (\ref{eq15}) and (\ref{eq16}) yields
\begin{equation}\label{eq39}
\gamma+\varepsilon^2\leq\Phi(\hat{f}(\bar{t}))+\Psi(\hat{f}(\bar{t}))\leq\gamma+\frac{5\varepsilon^2}{4}.
\end{equation}
Hence Theorem \ref{thm 5} is proved.

If $F$ is bounded and a closed subset of $X$, then by the definition
of $\delta$, we know that (\cite{7}) $\delta$ distance is equivalent to
the norm distance, so there is $c>0$ such that $dist_{\delta}(x,F_{\gamma})\geq c d(x,F_{\gamma})$,
where $d(x,F_r)$ is the distance between $x$ and the set $F_r$ deduced by the norm in the Banach space $X$.

Then we get
\begin{numcases}{}
\min_{x^*\in \partial\Phi(x)}\|x^*\|(1+\|x\|)\leq\frac{3\varepsilon}{2},\label{eq40}\\
d(x,F_{\gamma})\leq\frac{1}{c}\frac{3\varepsilon}{2},\\
\gamma\leq\Phi(x)\leq\gamma+\frac{5\varepsilon^2}{4}.
\end{numcases}
If we let $\varepsilon=\frac{1}{n}\rightarrow 0$, then we arrive at
a sequence $\{x_n\}$ which satisfies the requirements of Theorems
\ref{thm 6}, the Theorems \ref{thm 6} is proved. Theorems \ref{thm 7} and Theorem \ref{thm 8}
follow from Theorems \ref{thm 5} and Theorem \ref{thm 6}.

\end{proof}

\section{An Application to Hamiltonian systems}

Let $V\in C^{1-0}(\mathds{R}^n,\mathds{R})$; that is, $V$ is a locally Lipschitz potential function defined on $\mathds{R}^n$.
Let us consider the second order Hamiltonian systems
\begin{numcases}{}
-\ddot{q}(t)\in \partial V(q) \label{eq41}\\
\frac{1}{2}|\dot{q}|^2+V(q)=h\in \mathds{R}\label{eq42}
\end{numcases}
\begin{theorem}\label{thm 9}
Suppose $V\in C^{1-0}(\mathds{R}^n,\mathds{R})$ and $h\in \mathds{R}$
satisfy
  \begin{enumerate}
    \item[($V_1$)] $ V(-q)=V(q);$

    \item[($V_2$)] $\exists \mu_1>0,\mu_2\geq 0$, such that $\langle y,q\rangle\geq\mu_1V(q)-\mu_2,$
    $\forall y\in \partial V(q)$, $\forall q\in \mathds{R}^n$;

    \item[($V_3$)] there is an $ M>0$ such that $V(q)\geq h$, whenever $|q|\geq M $.
  \end{enumerate}
Then for any $h>\frac{\mu_2}{\mu_1},$ the system $(\ref{eq41})-(\ref{eq42})$ has at least
one non-constant periodic solution with the given energy $h$ which can be obtained by
Theorem \ref{thm 7}.
\end{theorem}

\begin{corollary}
For $a>0$, $\mu_1>0$, $\mu_2\geq0$, let $V(q)=a|q|^{\mu_1}+\frac{\mu_2}{\mu_1}$.
Then for any $h>\frac{\mu_2}{\mu_1},$ the system $(\ref{eq41})-(\ref{eq42})$ has at least
one non-constant periodic solution with the given energy $h$ which can be obtained by
Theorem \ref{thm 7}.
\end{corollary}

\begin{remark}If $\mu_1>1$, then $V(q)=a|q|^{\mu_1}+\frac{\mu_2}{\mu_1}\in C^1(\mathds{R}^n, \mathds{R})$;
but if $0<\mu_1\leq1$ it is not in $C^1(\mathds{R}^n,\mathds{R})$, but $V\in C^{1-0}(\mathds{R}^n,\mathds{R})$.
\end{remark}

In order to prove Theorem \ref{thm 9}, we define the Sobolev space
\begin{equation}\label{eq43}
H^1:=W^{1,2}(\mathds{R}/T\mathds{Z},\mathds{R}^n)=\{u:\mathds{R}\rightarrow \mathds{R}^n,u\in L^2,
\dot{u}\in L^2,u(t+1)=u(t)\}.
\end{equation}
 Then the standard $H^1$ norm is equivalent to
\begin{equation}\label{eq44}
 \|u\|:=\|u\|_{H^1}=\left(\int^1_0|\dot{u}|^2dt\right)^{1/2}+|\int_0^1 u(t)dt|.
 \end{equation}

\begin{lemma}\label{lemma1}{\rm(\cite{2})}\ \ Let
$f(u):=\frac{1}{2}\int^1_0|\dot{u}|^2dt\int^1_0(h-V(u))dt$ and
$\widetilde{u}\in H^1$ be such that $f^{\prime}(\widetilde{u})=0$
 and $f(\widetilde{u})>0$. Set
\begin{equation}\label{eq45}
\frac{1}{T^2}:=\frac{\int^1_0(h-V(\widetilde{u}))dt}{\frac{1}{2}\int^1_0|\dot{\widetilde{u}}|^2dt}.
\end{equation}
Then $\widetilde{q}(t)=\widetilde{u}(t/T)$ is a non-constant
$T$-periodic solution for {\rm(\ref{eq41})-(\ref{eq42})}.
\end{lemma}

In a manner similar to Ambrosetti-Coti
Zelati\cite{2}, from the symmetry condition $(V_1)$ we let
$E:=\{u\in
H^1=W^{1,2}(\mathds{R}/\mathds{Z},\mathds{R}^n),u(t+1/2)=-u(t)\}.$
A similar proof as in \cite{2}, we have

\begin{lemma}\label{lemma2}
 If $\bar{u}\in E$ is a critical point of
$f(u)$ and $f(\bar{u})>0$, then we have
$\bar{q}(t)=\bar{u}(t/T)$ is a
non-constant $T$-periodic solution of
{\rm(\ref{eq41})-(\ref{eq42})}.
\end{lemma}

We define a weakly closed subset of $H^1$,
\begin{equation}\label{eq46}
F_h:=\{u\in E:\int_0^1[V(u)+\frac{1}{2}\min\limits_{y\in \partial V(u)}\langle y, u\rangle] dt=h\}.
\end{equation}

\begin{lemma}\label{lemma3}
If $(V_2)-(V_3)$ hold, then $F_h\not=\emptyset$,for all $h > \frac{\mu_2}{\mu_1}$.
\end{lemma}

\begin{proof}
Take $u\in E$ satisfying $\min\limits_{t\in [0,1]}|u(t)|>0$. By condition $(V_2)$,
we know $V(0)\leq \frac{\mu_2}{\mu_1}$. We define
$$g(w)=\int_0^1[V(w)+\frac{1}{2}\min\limits_{y\in \partial V(w)}\langle y,w\rangle]dt,$$
$$g_u(a):=g(au)=\int_0^1[V(au)+\frac{1}{2}\min_{x\in \partial V(au)}\langle x, au\rangle]dt.$$
Then we have $g_u(0)=g(0)=V(0)\leq\frac{\mu_2}{\mu_1}. \label{eq47}$

We notice if $a$ is large enough, $\min\limits_{t\in [0,1]}|au|=a\min\limits_{0\leq t\leq 1} |u(t)|\geq M.$
We use $(V_2)-(V_3)$ to get
\begin{eqnarray}
g_u(a)=g(au)&=&\int_0^1[V(au)+\frac{1}{2}\min\limits_{x\in \partial V(au)}\langle x, au\rangle]dt\label{eq48}\\
&\geq&(1+\frac{\mu_1}{2})\int_0^1V(au)dt-\frac{\mu_2}{2}\label{eq49}\\
&\geq&(1+\mu_1)h-\frac{\mu_2}{2}=h+\frac{1}{2}(\mu_1h-\mu_2).
\end{eqnarray}
Hence any $ h>\frac{\mu_2}{\mu_1}$, when $a$ large enough, $g_u(a)>h$, we know there is $a(u)>0$ such that $a(u)u\in F$.
\end{proof}

\begin{lemma}\label{lemma4}
If $(V_1) - (V_3)$ hold, then for any given $c>0$,
 $f(u)$ satisfies $(CPS)_{F,c;\delta}$ condition; that is, if $\{u_n\}\subset E$ satisfies
\begin{equation}\label{eq51}
dist_\delta(u_n,F)\rightarrow 0,
 f(u_n)\rightarrow c>0,\ \ \ \
(1+\|u_n\|)\min_{y^*\in \partial f(u_n)}\|y^*\|\rightarrow 0,
\end{equation}
then $\{u_n\}$ has a strongly convergent subsequence.
\end{lemma}

\begin{proof}
 Notice that any $ u\in E, \int_0^1 u(t)dt=0$; hence,
we know $\|u\|_E:= (\int_0^1|\dot{u}|^2dt)^{1/2}$ is an
equivalent norm on $E$. By $f(u_n)\rightarrow c$, we have
\begin{equation}\label{eq52}
-\frac{1}{2}\|u_n\|^2_E\cdot\int^1_0V(u_n)dt\leq
c-\frac{h}{2}\|u_n\|_E^2+\varepsilon,
\end{equation}
where $\varepsilon \rightarrow 0$ when $n\rightarrow +\infty$.

By $(V_2)$ we know that any $ y^*\in \partial f(u_n)$, any $ x\in \partial V(u_n)$,

\begin{eqnarray}
\langle y^*,u_n\rangle &=&\|u_n\|_E^2\cdot\int^1_0[h-V(u_n)-\frac{1}{2}\langle x,u_n\rangle]dt\nonumber\\
&\leq&\|u_n\|_E^2\int^1_0[h+\frac{\mu_2}{2}-(1+\frac{\mu_1}{2})V(u_n)]dt.\label{eq53}
\end{eqnarray}
By (\ref{eq52}) and (\ref{eq53}) we have
\begin{eqnarray}
\langle y^*,u_n\rangle &\leq&(h+\frac{\mu_2}{2})\|u_n\|_E^2+(1+\frac{\mu_1}{2})(2c-h\|u_n\|_E^2)+(2+\mu_1)\varepsilon \nonumber\\
&=&(-\frac{\mu_1}{2}h+\frac{\mu_2}{2})\|u_n\|_E^2+\alpha ,\label{eq54}
\end{eqnarray}
where $\alpha=2(1+\frac{\mu_1}{2})(c+\varepsilon)$.

Since $h>\frac{\mu_2}{\mu_1}$, then (\ref{eq51}) and (\ref{eq54}) imply $\|u_n\|_{E}$
is bounded.

The rest of the argument to show $\{u_n\}$ has a strongly convergent subsequence
is standard.
\end{proof}
\begin{lemma}\label{lemma5}
Let
\begin{equation}\label{eq55}
G=\{u\in E:\int_0^1[V(u)+\frac{1}{2}\min_{y\in \partial V(u)}\langle y, u\rangle]dt<h\}.
\end{equation}
Then\\
(i). $F_h$ is the boundary of $G$.\\
(ii). If $(V_1)$ holds, then $F_h$ is symmetric with respect to
the origin $0$.\\
(iii). If $V(0)<h$ holds, then $0\in G$.
\end{lemma}

\begin{proof}
(i). By the definitions of $F$ and $G$.\\
(ii). By ($V_1$), we have $V(-u)=V(u)$, hence, $\partial V(-u)=-\partial V(u)$ and
\begin{eqnarray}
& &V(-u)+\frac{1}{2}\min\limits_{y\in \partial V(-u)}\langle y,-u\rangle\\
&=&V(u)+\frac{1}{2}\min\limits_{y\in -\partial V(u)}\langle -y,u\rangle\nonumber\\
&=&V(u)+\frac{1}{2}\min\limits_{-y\in \partial V(u)}\langle -y,u\rangle\nonumber\\
&=&V(u)+\frac{1}{2}\min\limits_{z\in \partial V(u)}\langle z,u\rangle\nonumber
\end{eqnarray}
By the definition of $F$ and the above equations, we know that $u\in F$ implies $(-u)\in F$.\\
(iii). If we set $W(u)=\int_0^1[V(u)+\frac{1}{2}\min\limits_{y\in \partial V(u)}\langle y,u\rangle]dt$, then
 $$W(0)=\int_0^1[V(0)+0]dt=V(0).$$
 Hence if $V(0)<h$, then $W(0)<h$ and $0\in G$.
\end{proof}

It's not difficult to prove the following two Lemmas:
\begin{lemma}\label{lemma6}
 $f(u)$ is weakly lower semi-continuous on $F_h.$
\end{lemma}

\begin{lemma}\label{lemma7}
$F_h$ is weakly closed subsets in $H^1$.
\end{lemma}

\begin{lemma}\label{lemma8}
 If $h>\frac{\mu_2}{\mu_1}$, then the functional $f(u)$ has
positive lower bound on $F_h$.
\end{lemma}

\begin{proof}
 By the definitions of $f(u)$ and $F_h$, we have
\begin{equation}\label{eq56}
f(u)=\frac{1}{4}\int^1_0|\dot{u}|^2dt\cdot\int^1_0\min_{y\in \partial V(u)}\langle y, u\rangle dt, \ \ u\in F_h.
\end{equation}
For $u\in F_h$ and $(V_2)$, we have
\begin{equation}\label{eq57}
\frac{1}{2}\int^1_0\min_{y\in \partial V(u)}\langle y, u\rangle dt
=\int^1_0 [ h-V(u)]dt\geq \int^1_0 [h-\frac{1}{\mu_1}\min_{y\in \partial V(u)}\langle y, u\rangle-\frac{\mu_2}{\mu_1}]dt,
\end{equation}
\begin{equation}\label{eq58}
\int^1_0\min_{y\in \partial V(u)}\langle y, u\rangle dt \geq\frac{h-\frac{\mu_2}{\mu_1}}{\frac{1}{2}+\frac{1}{\mu_1}}>0.
\end{equation}
So we have  $f(u)\geq 0$ for all $u\in F_h$. Furthermore, we claim that $\inf f(u)>0\label{eq59}$,
suppose otherwise, and note that $u(t)=const$ attains the infimum of 0.

 If $u\in F_h$, and $u$ is constant, then by the
symmetry $u(t+1/2)=-u(t)$ or $u(-t)=-u(t)$, we know $u(t)=0,\forall t$. By ($V_2$)
we have $V(0)\leq\frac{\mu_2}{\mu_1}$, by $h>\frac{\mu_2}{\mu_1}$ we get
$V(0)<h$.
From the definition of $F_h$, $0\notin F_h$. So $\inf_{F_h} f(u)>0.\label{eq60}$
Now by Lemmas \ref{lemma6}-\ref{lemma8}, we know $f(u)$ attains the
infimum on $F_h$, and the minimizer is nonconstant.
\end{proof}

\begin{lemma}\label{lemma9}
$\exists z_1\in H^1$ such that $z_1 \not=0$ and $f(z_1)\leq 0.$
\end{lemma}

\begin{proof}
We can choose $y_1(t)\in C[0,1]$ such that
 $\min\limits_{0\leq t\leq 1}|y_1(t)|>0$.
Let $z_1(t)=Ry_1(t)$, then when $R$ is large enough,
 by condition $(V_3)$ we have
\begin{equation}\label{eq61}
\int_0^1(h-V(z_1))dt\leq0;
\end{equation}
that is, $f(z_1)\leq 0.\label{eq62}$
\end{proof}

\begin{lemma}\label{lemma10}
 $f(0)=0.$
\end{lemma}

\begin{lemma}\label{lemma11}
 $F_h$ separates $z_1$ and $0$.
\end{lemma}

\begin{proof}
 By $V(0)<h$, we have that $0\in G$.
By $(V_2)$ and $(V_3)$ and $h>\frac{\mu_2}{\mu_1}$, we can choose $R$ large enough
such that
\begin{equation}
 z_1=Ry_1\in \{u\in H^1:\int_0^1[V(u)+\frac{1}{2}\min_{y\in \partial V(u)}\langle y, u\rangle]dt>h\}\label{eq63}
\end{equation}
In fact, since we can choose $y_1\in C[0,1]$ such that $\min\limits_{0\leq t\leq 1}|y_1(t)|>0$, hence, when $R$ large enough,
\begin{eqnarray}
 & &\int_0^1[V(Ry_1)+\frac{1}{2}\min\limits_{y\in\partial V(Ry_1)}\langle y,Ry_1\rangle]dt\\
 &\geq&(1+\frac{\mu_1}{2})\int_0^1V(Ry_1)dt-\frac{\mu_1}{2}\label{eq64}\\
 &\geq&(1+\frac{\mu_1}{2})h-\frac{\mu_1}{2}>h.\label{eq65}
\end{eqnarray}
Hence, $z_1=Ry_1\in\{u\in H^1: \int_0^1[V(u)+\frac{1}{2}\min\limits_{y\in\partial V(u)}\langle y,u\rangle]dt>h\}$.
So $F_h$ separates $z_1$ and $0$.
\end{proof}

Theorem \ref{thm 9} now follows from Theorem \ref{thm 7}

\section{Conclusions}

Since Ekeland's variational principle imposes less restriction on the functional,
we found it very useful in proving our Generalized Mountain Pass Lemma with weaker assumptions.
We were able to establish an immediate application for our Generalized Mountain Pass Lemma to
Hamiltonian systems  with Lipschitz potential and a fixed energy. It would be interesting to
see what role it can play for other differential equations.

\section*{Acknowledgements}

This research was partially supported by NSF of China(No.11701463, No.11671278 and No.12071316).
The authors sincerely thank the referees and the editors for their many valuable comments and
suggestions which help us improving the paper.

\bibliographystyle{unsrt}

\end{document}